\documentclass[12pt]{amsart}
\usepackage{amssymb,latexsym,esint}
\usepackage{enumerate}
\usepackage{float}
\usepackage{import}

\usepackage{mathtools}
\usepackage{amsfonts}
\usepackage{amssymb}
\usepackage{amsmath}
\usepackage{graphicx}
\usepackage{graphicx,color}
\usepackage[normalem]{ulem}
\usepackage{todonotes}

\usepackage{import}
\usepackage{xifthen}
\usepackage{pdfpages}
\usepackage{transparent}

\numberwithin{equation}{section}
\makeatletter
\@namedef{subjclassname@2020}{%
  \textup{2020} Mathematics Subject Classification}
\makeatother

\usepackage{amsmath,amstext,amsthm,amsxtra,mathtools,amssymb}

\usepackage[bookmarksopen,bookmarksdepth=3,colorlinks,citecolor=red,pagebackref,hypertexnames=true]{hyperref}
\usepackage[backrefs,msc-links,nobysame,initials,non-sorted-cites]{amsrefs}

\usepackage[normalem]{ulem}
\usepackage{lmodern}

\usepackage{lmodern}

\usepackage[backrefs,msc-links,nobysame,initials]{amsrefs}

\allowdisplaybreaks
\newtheorem{con}{Conjecture}
\newtheorem{prop}{Proposition}
\newtheorem{thm}{Theorem}

\theoremstyle{definition}

\theoremstyle{remark}

\begin{document}

\title[Sharp Weak Type Estimates for a Family of Zygmund Bases]{Sharp Weak Type Estimates for a family of Zygmund bases}
 \dedicatory{Dedicated to Alexey Solyanik on the occasion of his sixtieth birthday}


\author{Paul Hagelstein}
\address{P. H.: Department of Mathematics, Baylor University, Waco, Texas 76798}
\email{\href{mailto:paul_hagelstein@baylor.edu}{paul\_hagelstein@baylor.edu}}
\thanks{P. H. is partially supported by a grant from the Simons Foundation (\#521719 to Paul Hagelstein).}

\author{Alex Stokolos}
\address{A. S.: Department of Mathematical Sciences, Georgia Southern University, Statesboro, Georgia 30460}
\email{\href{mailto:astokolos@GeorgiaSouthern.edu}{astokolos@GeorgiaSouthern.edu}}


\begin{abstract}
Let $\mathcal{B}$ be a collection of rectangular parallelepipeds in $\mathbb{R}^3$ whose sides are parallel to the coordinate axes and such that $\mathcal{B}$ consists of parallelepipeds with side lengths of the form $s, 2^j s, t $, where $s, t > 0$ and $j$ lies in a nonempty subset $S$  of the integers.   In this paper, we prove the following:
\\

If $S$ is a finite set, then the associated geometric maximal operator $M_\mathcal{B}$ satisfies the weak type estimate of the form

$$\left|\left\{x \in \mathbb{R}^3 : M_{\mathcal{B}}f(x) > \alpha\right\}\right| \leq C \int_{\mathbb{R}^3} \frac{|f|}{\alpha}\left(1 + \log^+ \frac{|f|}{\alpha}\right)\;$$
but does not satisfy an estimate of the form

$$\left|\left\{x \in \mathbb{R}^3 : M_{\mathcal{B}}f(x) > \alpha\right\}\right| \leq C \int_{\mathbb{R}^3} \phi\left(\frac{|f|}{\alpha}\right)$$
for any convex increasing function $\phi: \mathbb[0, \infty) \rightarrow [0, \infty)$ satisfying the condition
$$\lim_{x \rightarrow \infty}\frac{\phi(x)}{x (\log(1 + x))} = 0\;.$$

  On the other hand, if $S$ is an infinite set, then the associated geometric maximal operator $M_\mathcal{B}$ satisfies the weak type estimate

$$\left|\left\{x \in \mathbb{R}^3 : M_{\mathcal{B}}f(x) > \alpha\right\}\right| \leq C \int_{\mathbb{R}^3} \frac{|f|}{\alpha} \left(1 + \log^+ \frac{|f|}{\alpha}\right)^{2}$$

\noindent but does not satisfy an estimate of the form

$$\left|\left\{x \in \mathbb{R}^3 : M_{\mathcal{B}}f(x) > \alpha\right\}\right| \leq C \int_{\mathbb{R}^3} \phi\left(\frac{|f|}{\alpha}\right)$$
for any convex increasing function $\phi: \mathbb[0, \infty) \rightarrow [0, \infty)$ satisfying the condition
$$\lim_{x \rightarrow \infty}\frac{\phi(x)}{x (\log(1 + x))^2} = 0\;.$$

\end{abstract}

\maketitle
 
\section{Introduction}

\emph{Weak type estimates} play a fundamental role in harmonic analysis.   A prevalent example of a weak type estimate is one satisfied by the Hardy-Littlewood maximal operator $M_{HL}$.  Recall that this operator is defined on $L^1(\mathbb{R}^n)$ by
$$M_{HL}f(x) = \sup_{x \in B}\frac{1}{|B|}\int_B |f|\;,$$
where the supremum is over balls $B$ in $\mathbb{R}^n$ containing $x$.   $M_{HL}$ satisfies the weak type $(1,1)$ inequality
$$\left|\left\{x \in \mathbb{R}^n : M_{HL}f(x) > \alpha\right\}\right| \leq C_n \frac{1}{\alpha} \int_{\mathbb{R}^n}|f|\;.$$

The \emph{strong maximal operator} $M_{str}$ is defined on $L^1(\mathbb{R}^n)$ by
$$M_{str} f(x) = \sup_{x \in R}\frac{1}{|R|}\int_R |f|\;,$$
where the supremum is over all rectangular parallelepipeds in $\mathbb{R}^n$ whose sides are parallel to the coordinate axes.   $M_{str}$ does not enjoy a weak type $(1,1)$ inequality; instead it satisfies the weak type $(L (\log L)^{n-1}, L^1)$  estimate
\begin{equation}\label{e000}\left|\left\{x \in \mathbb{R}^n : M_{str} f(x) > \alpha\right\}\right| \leq C_n \int_{\mathbb{R}^n}\frac{|f|}{\alpha}\left(1 + \log^{+}\frac{|f|}{\alpha}\right)^{n-1}\;.\end{equation}
Proofs of these estimates may be found in, e.g., \cite{guzman}.

Given a collection of sets $\mathcal{B}$ in $\mathbb{R}^n$, we may define the associated maximal operator $M_\mathcal{B}$ by
$$M_\mathcal{B}f(x) = \sup_{x \in R \in \mathcal{B}} \frac{1}{|R|}\int_R |f|\;.$$   If $\mathcal{B}$ is a proper subset of the collection of rectangular parallelepipeds, we refer to $\mathcal{B}$ as a \emph{rare basis} of parallelepipeds.   We would expect smaller collections $\mathcal{B}$ to be associated to better optimal weak type estimates for $M_\mathcal{B}$.  Indeed, if $\mathcal{B}$ were the collection of $n$-dimensional cubes in $\mathbb{R}^n$ we would have that $M_\mathcal{B}$ behaves like the Hardy-Littlewood maximal operator and satisfies a weak type $(1,1)$ inequality; if $M_\mathcal{B}$ were the collection of all rectangular parallelepipeds in $\mathbb{R}^n$ with sides parallel to the axes we would have that $M_\mathcal{B}$ is the strong maximal operator $M_{str}$ and satisfies the $(L (\log L)^{n-1}, L^1)$ estimate indicated above.   It is the case where $\mathcal{B}$ is an intermediate collection that most interests us here.

The most satisfying result to date along these lines is due to Stokolos, who in \cite{stokolos1988} (see also \cite{stokolos2005, stokolos2006}) proved the following.

\begin{prop}\label{prop2}
Let $\mathcal{B}$ be a translation invariant basis of rectangles in $\mathbb{R}^2$ whose sides are parallel to the coordinate axes.  If $\mathcal{B}$ does not satisfy the weak type $(1,1)$ estimate
$$|\{x \in \mathbb{R}^2 : M_\mathcal{B} f(x) > \alpha\}| \leq C \int_{\mathbb{R}^2} \frac{|f|}{\alpha}$$
then $M_\mathcal{B}$ satisfies the weak type estimate
$$\left|\left\{x \in \mathbb{R}^2 : M_\mathcal{B} f(x) > \alpha\right\}\right| \leq C \int_{\mathbb{R}^2} \frac{|f|}{\alpha} \left(1 + \log^+  \frac{|f|}{\alpha}\right)\;$$
but does not satisfy a weak type estimate of the form 
$$|\{x \in \mathbb{R}^2 : M_\mathcal{B} f(x) > \alpha\}| \leq C \int_{\mathbb{R}^2} \phi\left(\frac{|f|}{\alpha}\right)$$
for any nonnegative convex increasing function $\phi$ such that $\phi(x) = o(x \log x)$ as x tends to infinity.
\end{prop}

This result tells us that there is a certain ``discreteness'' associated to optimal weak type estimates for maximal operators associated to rare bases of rectangles in $\mathbb{R}^2$; the optimal weak type estimate must be of weak type $(1,1)$ or of weak type $(L \log L, L^1)$, but not of the form, say, $(L(\log L)^{1/2}, L^1)\;.$

At the present time there are no satisfactory analogues of Proposition \ref{prop2} for rare bases of rectangular parallelepipeds whose sides are parallel to the coordinate axes in $\mathbb{R}^n$ for $n \geq 3$.  We conjecture, however, the following.

\begin{con}\label{con1}
Let $\mathcal{B}$ be a translation invariant collection of rectangular parallelepipeds in $\mathbb{R}^n$ whose sides are parallel to the coordinate axes.  Then there exists an \emph{integer} $1 \leq k \leq n$ such that $M_\mathcal{B}$ satisfies the weak type estimate

\begin{equation}\label{e0}\left|\left\{x \in \mathbb{R}^n : M_\mathcal{B}f(x) > \alpha\right\}\right| \leq C \int_{\mathbb{R}^n} \frac{|f|}{\alpha}\left(1 + \log^+ \frac{|f|}{\alpha}\right)^{k-1}\end{equation} but such that $M_\mathcal{B}$ satisfies no estimate of the form

\begin{equation}\label{e1}\left|\left\{x \in \mathbb{R}^n : M_\mathcal{B}f(x) > \alpha\right\}\right| \leq C \int_{\mathbb{R}^n} \phi\left(\frac{|f|}{\alpha}\right)\end{equation}
whenever $\phi: [0, \infty) \rightarrow [0, \infty)$ is a convex increasing function satisfying $\phi(x) = o(x (\log x)^{k-1})$.
\end{con}


A significant contributor to the lack of progress on Conjecture \ref{con1} in dimensions $n=3$ and higher is the relative lack of meaningful classes of rare bases exhibiting known sharp weak type estimates.   In \cite{dm2017}, D'Aniello and Moonens provided a sufficient condition on a basis $\mathcal{B}$ of rectangular parallelepipeds in $\mathbb{R}^n$ so that estimate (\ref{e0}) is optimal for $k = n$.  Using the Fubini theorem and the weak type estimate (\ref{e000}) for the strong maximal operator one can  show that if the basis $\mathcal{B}$ consists of all rectangular parallelepipeds in $\mathbb{R}^3$ with sidelengths of the form $s, s, t$ the associated maximal operator $M_\mathcal{B}$ satisfies estimate (\ref{e0}) for $k=2$ but not estimate (\ref{e1}) for any nonnegative convex increasing function $\phi$ satisfying 
$\phi(x) = o(x (\log x))$.  (See a related discussion of this basis in the seminal paper \cite{zygmund1967} of Zygmund which initiated the topic of rare bases in the subject of differentiation of integrals.) In the more involved argument in \cite{soria}, Soria proved that if $\mathcal{B}$ consists of all rectangular parallelepipeds in $\mathbb{R}^3$ with sidelengths of the form $s, \frac{1}{s}, t$, then $M_\mathcal{B}$ also satisfies estimate (\ref{e0}) for $k=2$ but not estimate (\ref{e1}) for any nonnegative convex increasing function $\phi$ satisfying 
$\phi(x) = o(x (\log x))$. \footnote{ In was in this same paper that Soria disproved the \emph{Zygmund Conjecture}. \cite{fefbeijing} provides a good introduction to the Zygmund Conjecture for the interested reader.   A recent class of counterexamples to the Zygmund Conjecture due to Rey may be found in \cite{rey2020}.  Important contexts where the Zygmund Conjecture does hold are due to A. C\'ordoba \cite{cordoba}; extensions of C\'ordoba's associated covering lemma techniques to higher dimensions due to R. Fefferman and  Pipher may be found in \cite{fp2005}.}   The most recent result to date along these lines is due to Dmitrishin, Hagelstein, and Stokolos, who proved in \cite{dms2021} that if  $\mathcal{B}$ be a collection of rectangular parallelepipeds in $\mathbb{R}^3$ whose sides are parallel to the coordinate axes and such that $\mathcal{B}$ contains parallelepipeds with side lengths of the form $s, \frac{2^N}{s} , t $, where $s, t > 0$ and $N$ lies in an infinite subset of the integers, then the associated geometric maximal operator $M_\mathcal{B}$ satisfies the weak type estimate (\ref{e0}) for $k=3$ but does not satisfy the estimate (\ref{e1}) for any nonnegative convex increasing function $\phi$ satisfying 
$\phi(x) = o(x (\log x)^2)$.   The argument in the latter paper is quite delicate and utilizes the concept of \emph{crystallization}, introduced by Stokolos in \cite{stokolos1988} and developed further in \cite{dms2021, hs2011, stokolos2005,stokolos2006}.  

The purpose of this paper is to provide another  class of natural examples of rare bases in $\mathbb{R}^3$ for which Conjecture \ref{con1} holds.   The associated proof also involves crystallization, but in many respects it is more straightforward than the argument in \cite{dms2021} as we are here able to exploit the \emph{homothecy invariance} of the bases considered.  It is our hope that the theorem is not only of intrinsic interest, but the techniques of proof might be used in the future to help resolve Conjecture \ref{con1} in the special but important case of homothecy invariant bases.

Our main result is the following.

\begin{thm}\label{t1}
Let $\mathcal{B}$ be a homothecy invariant collection of rectangular parallelepipeds in $\mathbb{R}^3$ whose sides are parallel to the coordinate axes and with sidelengths of the form $s, 2^j s, t$, where $j$ lies in a nonempty set $S \subset \mathbb{Z}$.  

If $S$ is a finite set, then the associated geometric maximal operator $M_\mathcal{B}$ satisfies the weak type estimate of the form

\begin{equation} \label{e2}\left|\left\{x \in \mathbb{R}^3 : M_{\mathcal{B}}f(x) > \alpha\right\}\right| \leq C \int_{\mathbb{R}^3} \frac{|f|}{\alpha}\left(1 + \log^+ \frac{|f|}{\alpha}\right)\;\end{equation}
but does not satisfy an estimate of the form

\begin{equation}\label{e5}\left|\left\{x \in \mathbb{R}^3 : M_{\mathcal{B}}f(x) > \alpha\right\}\right| \leq C \int_{\mathbb{R}^3} \phi\left(\frac{|f|}{\alpha}\right)\end{equation}
for any convex increasing function $\phi: \mathbb[0, \infty) \rightarrow [0, \infty)$ satisfying the condition
\begin{equation}\label{e6}\lim_{x \rightarrow \infty}\frac{\phi(x)}{x (\log(1 + x))} = 0\;.\end{equation}
\\

If $S$ is an infinite set, then the associated geometric maximal operator $M_\mathcal{B}$ satisfies a weak type estimate of the form

$$\left|\left\{x \in \mathbb{R}^3 : M_{\mathcal{B}}f(x) > \alpha\right\}\right| \leq C \int_{\mathbb{R}^3} \frac{|f|}{\alpha} \left(1 + \log^+ \frac{|f|}{\alpha}\right)^{2}$$

\noindent but does not satisfy an estimate of the form

$$\left|\left\{x \in \mathbb{R}^3 : M_{\mathcal{B}}f(x) > \alpha\right\}\right| \leq C \int_{\mathbb{R}^3} \phi\left(\frac{|f|}{\alpha}\right)$$
for any convex increasing function $\phi: \mathbb[0, \infty) \rightarrow [0, \infty)$ satisfying the condition
$$\lim_{x \rightarrow \infty}\frac{\phi(x)}{x (\log(1 + x))^2} = 0\;.$$
\end{thm}

The remainder of the paper is devoted to a proof of this theorem.  We remark that the statement of Theorem \ref{t1} is very similar to that of Theorem 1 of \cite{dms2021} but there are nonetheless significant differences between these two results.   In particular, the basis $\mathcal{B}$ in the latter paper consists of rectangular parallelepipeds parallel to the coordinate axes with sidelengths of the form $s, \frac{2^N}{s}, t$ where $N$ lies in a nonempty set $S$ of integers.   To the best of our understanding neither result follows from the other, in large part because the latter basis lacks the \emph{dilation invariance} enjoyed by the former as well as the fact that, even for finite sets $S$, the possible ratios of the first two sidelengths of parallelepipeds in the latter basis are all of $(0,\infty)$.
\\

{{\bf{Acknowledgment:} }  We wish to thank the referee for helpful suggestions regarding this paper.}
\section{Proof of Theorem \ref{t1}}
\begin{proof}[Proof of Theorem \ref{t1}]

If $S$ is finite, we may itemize the elements of $S$ as $j_1 , \ldots, j_N$.  Subsequently we may express $M_\mathcal{B}$ as a supremum of maximal functions of the form $M_{\mathcal{B}_k}$, where $\mathcal{B}_k$ consists of all parallelepipeds in $\mathbb{R}^3$ with sidelengths of the form $s, 2^{j_k}s, t$.  From the paper \cite{fava1972} of Fava it readily follows that estimate (\ref{e2}) is satisfied.  As can be seen by testing the maximal operator $M_\mathcal{B}$ on the characteristic functions associated to cubes in $\mathbb{R}^3$, one can readily show that estimate (\ref{e5}) does not hold for any convex increasing function $\phi:[0,\infty) \rightarrow [0, \infty)$ satisfying the limit (\ref{e6}).
\\

We now turn to case that $S$ is an infinite set.   First we note that $M_\mathcal{B}$  satisfies the estimate 
$$\left|\left\{x \in \mathbb{R}^3 : M_{\mathcal{B}}f(x) > \alpha\right\}\right| \leq C \int_{\mathbb{R}^3} \frac{|f|}{\alpha} \left(1 + \log^+ \frac{|f|}{\alpha}\right)^{2}$$ as it is dominated by the strong maximal operator $M_{str}$ acting on measurable functions in $\mathbb{R}^3$. 

 It remains to show that
 $M_\mathcal{B}$ does not satisfy an estimate of the form $$\left|\left\{x \in \mathbb{R}^3 : M_{\mathcal{B}}f(x) > \alpha\right\}\right| \leq C \int_{\mathbb{R}^3} \phi\left(\frac{|f|}{\alpha}\right)$$
for any convex increasing function $\phi: \mathbb[0, \infty) \rightarrow [0, \infty)$ satisfying the condition
$$\lim_{x \rightarrow \infty}\frac{\phi(x)}{x (\log(1 + x))^2} = 0\;.$$  This is the primary difficulty we face. Note that we cannot show that $M_\mathcal{B}$ fails to satisfy such an estimate simply by testing $M_{\mathcal{B}}$ on the characteristic function of a  cube.   Instead we need to utilize ideas involving \emph{crystallization} introduced by Stokolos in \cite{stokolos1988} and further developed by Dmitrishin, Hagelstein and Stokolos in \cite{dms2021, hs2011, stokolos2005, stokolos2006}.
\\

Let $N > 1$ be a positive integer.  Since $S$ is an infinite set contained in the integers, we may assume without loss of generality that there exist natural numbers \mbox{$j_1 > j_2 > \cdots > j_N$} so that any parallelepiped in $\mathbb{R}^3$ whose sides are parallel to the coordinate axes with sidelengths of the form $ 2^{-k + 1}, 2^{-j_k}, t$ lie in $\mathcal{B}$.  
 Moreover for technical reasons we will see later we also assume without loss of generality that $j_k > N + j_{k-1}$.  Let 
\begin{align}
R_1 &= [0,1] \times [0, 2^{-j_1}],\notag
\\R_2 &= [0, \frac{1}{2}] \times [0, 2^{-j_2}], \notag
\\ &\vdots  \notag 
\\R_N &= [0, 2^{-N+1}] \times [0, 2^{-j_N}]\;.\notag
\end{align}

 We define the Rademacher function $r_0(t)$ by
 $$r_0(t) = \chi_{[0, \frac{1}{2}]}(t) - \chi_{(\frac{1}{2},1)}(t)\;,$$ where we extend $r_0(t)$ to be periodic by $r_0(t+1) = r_0(t)\;.$   Let $C_{j_1, \ldots, j_N}$ be defined by
 $$C_{j_1, \ldots, j_N} = \left\{t \in [0,1] : \sum_{k=1}^N r_0(2^{j_k}t) = N\right\}\;.$$
 Observe that $m_1(C_{j_1, \ldots, j_N}) = 2^{-N}$, where we let $m_k(A)$ denote the $k$-dimensional Lebesgue measure of a set $A \subset \mathbb{R}^k$. (That this holds can   be seen by noting that the sets $\{t \in [0,1]: r_0(2^{j_k}t) = 1\}$ correspond to $N$ mutually independent events all of probability $\frac{1}{2}$.)  Define now the set $E_{j_1, \ldots, j_N}$ (which we will abbreviate as $E_N$) in $\mathbb{R}^2$ by
 $$E_{j_1, \ldots, j_N} = [0, 2^{-N+1}] \times C_{j_1, \ldots, j_N}\;.$$ 
\\
\begin{figure}[ht]
\centering
\def\svgwidth{300pt}
\begingroup%
  \makeatletter%
  \providecommand\color[2][]{%
    \errmessage{(Inkscape) Color is used for the text in Inkscape, but the package 'color.sty' is not loaded}%
    \renewcommand\color[2][]{}%
  }%
  \providecommand\transparent[1]{%
    \errmessage{(Inkscape) Transparency is used (non-zero) for the text in Inkscape, but the package 'transparent.sty' is not loaded}%
    \renewcommand\transparent[1]{}%
  }%
  \providecommand\rotatebox[2]{#2}%
  \newcommand*\fsize{\dimexpr\f@size pt\relax}%
  \newcommand*\lineheight[1]{\fontsize{\fsize}{#1\fsize}\selectfont}%
  \ifx\svgwidth\undefined%
    \setlength{\unitlength}{299.61127461bp}%
    \ifx\svgscale\undefined%
      \relax%
    \else%
      \setlength{\unitlength}{\unitlength * \real{\svgscale}}%
    \fi%
  \else%
    \setlength{\unitlength}{\svgwidth}%
  \fi%
  \global\let\svgwidth\undefined%
  \global\let\svgscale\undefined%
  \makeatother%
  \begin{picture}(1,1)%
    \lineheight{1}%
    \setlength\tabcolsep{0pt}%
    \put(0,0){\includegraphics[width=\unitlength,page=1]{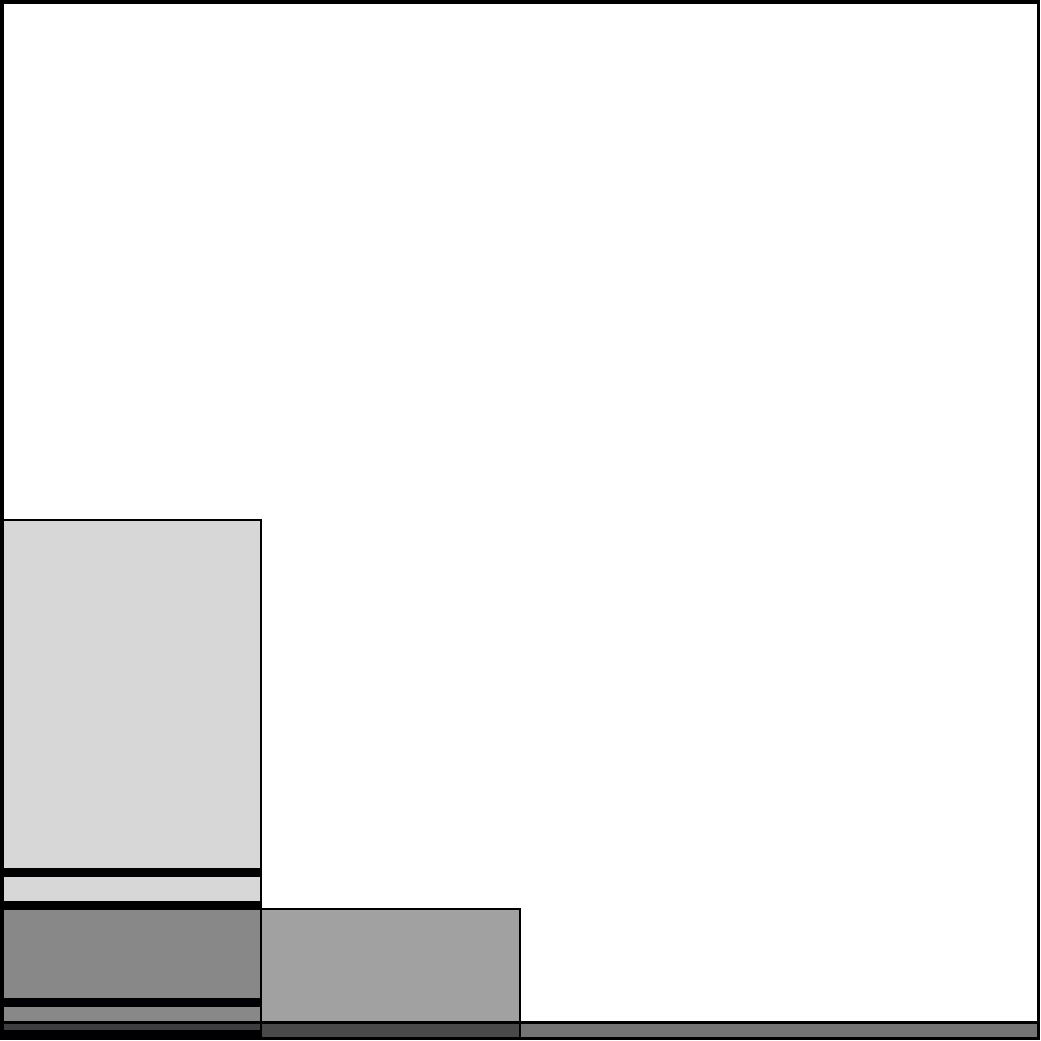}}%
    \put(0.28200542,0.40657373){\color[rgb]{0,0,0}\makebox(0,0)[lt]{\lineheight{1.25}\smash{\begin{tabular}[t]{l}$R_3$\end{tabular}}}}%
    \put(0.53114208,0.14186605){\color[rgb]{0,0,0}\makebox(0,0)[lt]{\lineheight{1.25}\smash{\begin{tabular}[t]{l}$R_2$\end{tabular}}}}%
    \put(0.20415019,0.73356564){\color[rgb]{0,0,0}\makebox(0,0)[lt]{\lineheight{1.25}\smash{\begin{tabular}[t]{l}$E_{5,3,1}$\end{tabular}}}}%
    \put(0.82699191,0.04843978){\color[rgb]{0,0,0}\makebox(0,0)[lt]{\lineheight{1.25}\smash{\begin{tabular}[t]{l}$R_1$\end{tabular}}}}%
    \put(0,0){\includegraphics[width=\unitlength,page=2]{p45figddesktop.pdf}}%
  \end{picture}%
\endgroup%

\caption{The set $E_{5,3,1}$ together with rectangles $R_1$, $R_2$, and $R_3$}
\label{aa3}
\end{figure}

Figure 1 provides an illustration of the case that $j_1 = 5$, $j_2 = 3$, and $j_3 = 1$, indicating the set $E_{5,3,1}$ together with associated rectangles $R_1$, $R_2$, and $R_3$.
\\

 We now assume without loss of generality that $j_{k-1} > N + j_{k}$. 
 \\
 
  Note  that
 
 $$\frac{1}{m_2(R_k)}\int_{R_k}\chi_{E_N} = 2^{-N}$$ for each $k$.
  For each $k$, we let $\{R_{k,j}\}_j$ be the collection of all the vertical translates $ R^\prime$ in $\mathbb{R}^2$ of the rectangle $R_k$ that are dyadic rectangles and such that
 
$$\frac{1}{m_2(R^\prime)}\int_{R^\prime} \chi_{E_N} = 2^{-N}\;.$$ 
Note that the number of such translates is  
\begin{align}
\#\{R_{k,j}\} &= \frac{m_2(E_N)}{m_2(E_N \cap R_k)} \notag
\\&= \frac{2^{-N+1}\cdot 2^{-N}}{2^{-N+1}\cdot\left(2^{-k}\cdot 2^{-j_k}\right)}\notag
\\&= 2^{-N +k + j_k}\;\notag
\end{align}
 and moreover these translates, for fixed $k$, are a.e. pairwise disjoint.
 \\
 
 It will now be convenient to have a notation indicating homothecies of a rectangle $R$ that share the lower left corner of $R$.  This is done as follows.  Given a rectangle $R = [a, a + \Delta_1]\times[b, b + \Delta_2]\subset \mathbb{R}^2$ and $\delta > 0$, 
 we let
 $$\delta R = [a, a + \delta \Delta_1]\times[b, b + \delta\Delta_2]\;.$$
 
 Given $s \in \mathbb{R}$, let $\tau_s R$ denote the vertical translate of the rectangle $R \subset \mathbb{R}^2$ given by
 $$\chi_{\tau_s R}(x_1, x_2) = \chi_R(x_1, x_2 - s)\;.$$
 Here $\chi_A$ denotes the usual indicator function of set $A$.
 Given $k$, $1 \leq j \leq 2^{-N + k + j_k}$, a nonnegative integer $l$, and $1 \leq i \leq 2^l$, we define  $R_{k, j, 2^l, i}$ (a translate of a dilate of $R_k$) by
 
 $$R_{k, j, 2^l, i} = \tau_{(i-1)\cdot 2^{-l}\cdot 2^{-j_k}}(2^{-l}R_{k,j})\;.$$
(Note that we can recognize that $-N + k + j_k$ is nonnegative since the conditions \mbox{$j_1 > j_2 > \cdots > j_N$} together with the fact that each $j_i$ is a natural number guarantee that $j_k \geq N - k + 1$.)
 One can compute that
 
 $$\frac{1}{m_2(R_{k, j, 2^0, 1})}\int_{R_{k, j, 2^0, 1}} \chi_{E_N} = 2^{-N}\;,$$
 $$\frac{1}{m_2(R_{k, j, 2, 1})}\int_{R_{k, j, 2^1, 1}} \chi_{E_N} = 2^{-N+2}\;,$$
 and more generally
 $$\frac{1}{m_2(R_{k,j,2^l, i})}\int_{R_{k,j,2^l, i}}\chi_{E_N} = 2^{-N + l + 1}$$
 for $i = 1, \ldots, 2^{l-1}$, \emph{provided} $1 \leq l$  and $2^{-N + 1} \leq 2^{-l}\cdot 2^{-k+1}$, i.e., provided $1\leq l \leq N-k$.  It is for these estimates that we require the sparseness condition $j_{k-1} > N + j_{k}$, as these enable us to have that, for $1 \leq l \leq N-k$ and fixed $j,k$,  the intersections of $E_N$ with the left hand sides of the  $R_{k,j,2^l, i}$ are vertical translates of one another for $1 \leq i \leq 2^{l-1}$.
 \\
 
 It is important here to recognize that each $R_{k, j, 2^l, i}$ is indeed a translate of a dilate of  $R_{k}$, so rectangular parallelepipeds of the form $R_{k,j,2^l, i}\times [0,t]$ lie in the basis $\mathcal{B}$ since $\mathcal{B}$ is \emph{homothecy invariant.}   This is a point where the argument we provide here differs significantly from the crystallization argument provided in the related paper \cite{dms2021}.  
 \\

 Define now the set $Z_{j_1, \ldots, j_N}$ (which we will refer to simply as $Z_N$) in $\mathbb{R}^3$ by
 $$Z_{j_1, \ldots, j_N} = E_{j_1, \ldots, j_N} \times [0,1]\;.$$
 
 Given a rectangle $R \subset \mathbb{R}^2$ whose sides are parallel to the coordinate axes, we let $rh(R)$ denote the right half of $R$.
 
 The above averages over the $R_{j,k,2^l, i}$ yield the inclusion

$$ \left\{x \in \mathbb{R}^3 : M_\mathcal{B}\chi_{Z_N}(x) \geq 2^{-N} \right\} \supset \bigcup_{j,k,2^l,i \atop 1 \leq l \leq N - k}^\cdot (rh(R_{k,j,2^l,i})) \times [2^{l}, 2^{l+1}]\;,$$
 where the sets in the union are a.e. pairwise disjoint.
 This yields that
 
 \begin{align}
 \left| \left\{x \in \mathbb{R}^3 : M_\mathcal{B}\chi_{Z_N}(x) \geq 2^{-N} \right\} \right| \notag &\geq \sum_{k = 1}^N \sum_{j = 1}^{2^{-N + k + j_k}}\sum_{l=1}^{N-k}\sum_{i = 1}^{ 2^{l-1}} \left|(rh(R_{k,j,2^l,i})) \times [2^{l}, 2^{l+1}]\right|\; 
 \\&= \sum_{k = 1}^N \sum_{j = 1}^{2^{-N + k + j_k}}\sum_{l=1}^{N-k}\sum_{i = 1}^{ 2^{l-1}}\frac{1}{2}2^{-2l}|R_k|\cdot 2^l \notag
  \\&= \sum_{k = 1}^N \sum_{j = 1}^{2^{-N + k + j_k}}(N-k)\cdot2^{-k-j_k-1}\notag
   \\&= \sum_{k = 1}^N  (N-k)\cdot 2^{-N + k + j_k}\cdot 2^{-k-j_k-1}\notag
   \\&\gtrsim \sum_{k=1}^N (N-k)\cdot 2^{-N}\notag
   \\&\gtrsim N^2 2^{-N}\;. \notag
     \end{align}
 
 As $|Z_N| = 2^{-2N + 1}$, we see $M_\mathcal{B}$ does not satisfy an estimate of the form 
 
 \begin{equation}\label{e21}\left|\left\{x \in \mathbb{R}^3 : M_{\mathcal{B}}f(x) > \alpha\right\}\right| \leq C \int_{\mathbb{R}^3} \phi\left(\frac{|f|}{\alpha}\right)
 \end{equation}
for any convex increasing function $\phi: \mathbb[0, \infty) \rightarrow [0, \infty)$ satisfying the condition
$$\lim_{x \rightarrow \infty}\frac{\phi(x)}{x (\log(1 + x))^2} = 0\;.$$ 
To see this, suppose $\phi$ satisfies (\ref{e21}).    Setting $f = \chi_{Z_N}$, $\alpha = 2^{-N}$ then yields that $$N^2 2^N \leq C \phi(2^N)\;,$$ providing the desired result.


\end{proof}


\begin{bibsection}
\begin{biblist}


\bib{favacapri}{article}{
author = {O. N. Capri},
author = {N. A. Fava},
journal = {Studia Math.},
volume = {78},
year = {1984},
title = {Strong differentiability with respect to product measures},
pages = {173--178},
review ={\MR{0766713}},
}

\bib{cordoba}{article}{
author = {A. C\'ordoba},
journal = {Harmonic analysis in Euclidean spaces (Proc. Sympos. Pure Math., Williams Coll., Williamstown, Mass., 1978) Part 1},
venue = {Williams Coll., Williamstown, Mass.}
volume = {35},
year = {1979},
title = {Maximal functions, covering lemmas and Fourier multipliers},
pages = {29--50},
review ={\MR{0545237}},
}


\bib{cf1975}{article}{
author = {A. C\'ordoba},
author = {R. Fefferman},
journal = {Ann. of Math.},
volume = {102},
year = {1975},
title = {A geometric proof of the strong maximal theorem},
pages = {95--100},
review={\MR{0379785}},
}


\bib{dm2017}{article}{
author = {E. D'Aniello},
author = {L. Moonens},
journal = {Ann. Acad. Sci. Fenn. Math.},
volume = {42},
year = {2017},
pages = {119--133},
title = {Averaging on $n$-dimensional rectangles},
review = {\MR{3558519}},
}


\bib{dms2021}{article}{
author = {D. Dmitrishin},
author = {P. Hagelstein},
author = {A. Stokolos},
title = {Sharp weak type estimates for a family of Soria bases},
journal = {submitted for publication},
eprint = {2101.08736},
}

\bib{fava1972}{article}{
author = {N. Fava},
journal = {Studia Math.},
volume = {42},
year = {1972},
title = {Weak type inequalities for product operators},
pages = {271--288},
review = {\MR{308364}},
}

\bib{fefbeijing}{article}{
author={R. Fefferman},
title={Multiparameter Fourier analysis},
journal={Beijing lectures in harmonic analysis (Beijing, 1984), Ann. of Math. Stud.}
volume={112},
pages={47--130}, 
publisher={Princeton Univ. Press},
review={\MR{0864655}},
}


\bib{fp2005}{article}{
author = {R. Fefferman},
author = {J. Pipher},
title = {A covering lemma for rectangles in $\mathbb{R}^n$},
journal = {Proc. Amer. Math. Soc.},
year = {2005},
volume={133},
pages = {3235--3241},
review = {\MR{2161145}},
}

\bib{guzman1974}{article}{
author = {M. de Guzm\'an},
journal = {Studia Math.},
volume = {49},
year = {1974},
pages = {188--194},
title = {An inequality for the Hardy-Littlewood maximal operator with respect to a product of differentiation bases},
review = {\MR{0333093}},
}

\bib{guzman}{book}{
author = {M. de Guzm\'an},
title = {Differentiation of integrals in $\mathbb{R}^n$},
series = {Lecture Notes in Mathematics},
volume = {481},
publisher = {Springer-Verlag},
year = {1975},
review = {\MR{0457661}},
}


\bib{hs2011}{article}{
author = {P. Hagelstein},
author = {A. Stokolos},
journal = {New York J. Math.},
volume = {17},
year = {2011},
title = {Weak type inequalities for maximal operators associated to double ergodic sums},
pages = {233--250},
review = {\MR{2781915}},
}


\bib{rey2020}{article}{
author={G. Rey},
title={Another counterexample to Zygmund's conjecture},
journal={Proc. Amer. Math. Soc.},
volume ={148},
year={2020},
pages={5269--5275},
review={\MR{4163839}},
}

\bib{soria}{article}{
author = {Soria, F.},
journal = {Ann. of Math.},
volume = {123},
title = {Examples and counterexamples to a conjecture in the theory of differentiation of integrals},
year = {1986},
pages = {1--9},
 review={\MR{0825837}},
}



\bib{stokolos1988}{article}{
author = {A. M. Stokolos},
journal = {Studia Math.},
volume = {88},
title = {On the differentiation of integrals of functions from $L \phi(L)$},
year = {1988}, 
pages = {103--120},
review = {\MR{931036}},
}

\bib{stokolos2005}{article}{
author = {A. M. Stokolos},
journal = {Ann. Inst. Fourier (Grenoble)},
title = {Zygmund's program: some partial solutions},
volume = {55},
year = {2005},
pages = {1439--1453}, 
review = {\MR{2172270}},
}

\bib{stokolos2006}{article}{
author = {Stokolos, A. M.},
journal = {Colloq. Math.}
title = {On weak type inequalities for rare maximal functions in $\mathbb{R}^n$},
volume = {104},
year = {2006},
pages = {311--315},
review = {\MR{2197080}},
}



\bib{zygmund1967}{article}{
author = {Zygmund, A.},
journal = {Colloq. Math.},
volume = {16},
year = {1967},
title = {A note on the differentiability of integrals},
pages = {199--204},
review = {\MR{0210847}},
}


\end{biblist}
\end{bibsection}
\end{document}